\documentclass[dvips,11pt,a4paper, oneside]{amsart}
\usepackage{amssymb}
\usepackage[latin2]{inputenc}
\usepackage{graphicx}
\usepackage{indentfirst}

\newtheorem{Def}{Definition}
\newtheorem{Lem}{Lemma}
\newtheorem{Prop}{Proposition}
\newtheorem{Thm}{Theorem}
\newtheorem{Cor}{Corollary}
\newtheorem{Rem}{Remark}
\newenvironment{Pf}{ Proof.}{\(\square\)}

\title[Average methods and their applications ...]{Average methods and their applications in differential geometry I}
\author{Cs. Vincze}
\address{Inst. of Math., Univ. of Debrecen \\
H-4010 Debrecen, P.O.Box 12 \\
Hungary}
\email{csvincze@science.unideb.hu}

\begin{document}
\begin{abstract}
In Minkowski geometry the metric features are based on a compact convex body containing the origin in its interior. This body works as a unit ball with its boundary formed by the unit vectors. Using one-homogeneous extension we have a so-called Minkowski functional to measure the lenght of vectors. The half of its square is called the energy function. Under some regularity conditions we can introduce an average Euclidean inner product by integrating the Hessian matrix of the energy function on the Minkowskian unit sphere. Changing the origin in the interior of the body we have a collection of Minkowskian unit balls together with Minkowski functionals depending on the base points. It is a kind of special Finsler manifolds called a Funk space. Using the previous method we can associate a Riemannian metric as the collections of the Euclidean inner products belonging to different base points. We investigate this procedure in case of Finsler manifolds in general. Central objects of the associated Riemannian structure will be expressed in terms of the canonical data of the Finsler space. We take one more step forward by introducing associated Randers spaces by averaging of the vertical derivatives of the Finslerian fundamental function. Such kind of perturbation will have a crucial role when we apply the general results to Funk spaces together with some contributions to Brickell's conjecture on Finsler manifolds with vanishing curvature tensor $\mathbb{Q}$. 
\end{abstract}
\maketitle
\section{An observation on homogeneous functions}

\noindent
I. Let $\mathbb{R}^n$ be the standard real coordinate space equipped with the canonical inner product. The canonical coordinates are denoted by $y^1,\ldots,y^n$. The set $B=\{v\in \mathbb{R}^n\ | \ (v^1)^2+\ldots(v^n)^2\leq 1\}$
is the unit ball with boundary $\partial B$ formed by Euclidean unit vectors. 

\vspace{0.2cm}
\noindent
II. The function
$\varphi\colon \mathbb{R}^n\to \mathbb{R}$ is called \emph{positively homogene\-ous of degree $k$} if $\varphi(tv)=t^k\varphi(v)$ for all $v\in \mathbb{R}^n$ and $t>0$. In case of $k\neq 0$ homogeneous functions are uniquely determined by the restriction to $\partial B$. If $k=0$ then the value at the origin is optional. The extension of a function on $\partial B$ to a positively homogeneous one heritages the analytical properties at each non-zero point of the space. The origin has a crucial role in the theory of homogeneous functions. One of the most important results is Euler's theorem which states that if $\varphi$ is differentiable away from the origin, then it is positively homogeneous of degree $k$ if and only if 
$C\varphi =k\varphi$, where
$$C:=y^1\frac{\partial}{\partial y^1}+\ldots+y^n\frac{\partial}{\partial y^n}$$
is the so-called Liouville vector field. Note that the partial derivatives of $\varphi$ are positively homogeneous functions of degree $k-1$. The analytic properties are reduced at the origin in general. It is known that if a function of class $C^k$ is positively homogeneous of degree $k\geq 0$ then it must be polynomial of degree $k$. 

\begin{Lem}
Suppose  that $\varphi$ is positively homogeneous of degree zero, differentiable away from the origin and $\varphi(v)>0$ for all $v\neq \bf{0}$. Then the mapping $T\colon v\mapsto T(v):=\varphi(v)v$
has the Jacobian $\det J_v T=\varphi^{n}(v).$
\end{Lem}

\begin{Pf}
Let $v$ be a non-zero element in $\mathbb{R}^n$ and consider the sequence of mappings $G_m(v):=h^{m}(v)v$, where $m=0, 1, \ldots$ and $h(v):=\ln \varphi(v)$. It can be easily seen that 
$$\frac{\partial G_1^{i}}{\partial y^j}=\delta_j^i h+y^i\frac{\partial h}{\partial y^j},$$
where $\delta_j^i$ is the usual Kronecker-symbol. We also have by the zero homoge\-ni\-ty of the function $h$ that 
$$\sum_{k=1}^n\bigg(\delta_k^i h+y^i\frac{\partial h}{\partial y^k}\bigg)\bigg(\delta_j^k h+y^k\frac{\partial h}{\partial y^j}\bigg)=\delta^i_jh^2+2y^ih\frac{\partial h}{\partial y^j}$$
which means that $\big(J_v G_1\big)^2=J_v G_2$. By a simple induction 
$$\big(J_v G_1\big)^{m}=\big(J_v G_1\big)^{m-1}J_v G_1=J_v G_{m-1} J_v G_1=J_v G_{m}$$
because of
$$\sum_{k=1}^n\bigg(\delta_k^i h^{m-1}+(m-1)y^ih^{m-2}\frac{\partial h}{\partial y^k}\bigg)\bigg(\delta_j^k h+y^k\frac{\partial h}{\partial y^j}
\bigg)=$$
$$=\delta^i_jh^m+my^ih^{m-1}\frac{\partial h}{\partial y^j}.$$
Therefore
$$\exp J_v G_1:=\sum_{m=0}^{\infty}\frac{\big(J_v G_1\big)^m}{m!}=\sum_{m=0}^{\infty}\frac{J_v G_m}{m!}=e^{h(v)}\bigg(\delta_j^i+
v^i\frac{\partial h}{\partial y^j}_v
\bigg)=J_v T$$ 
and thus $\det \exp J_v G_1=\det J_v T$.
Since 
$$e^{\textrm{trace} J_v G_1}=\det \exp J_v G_1,$$
we have by the zero homogenity of the function $h$ that 
$$\textrm{trace}\  J_v G_1=nh(v)+\sum_{k=1}^n v^k\frac{\partial h}{\partial y^k}_v=nh(v)$$
and, consequently, $\det J_v T=\varphi^{n}(v)$ as was to be proved.
\end{Pf}

\section{Minkowski functionals and associated objects}

\noindent
I. Let $K\subset \mathbb{R}^n$ be a convex body containing the origin in its interior and consider a non-zero element $v\in \mathbb{R}^n$. 
The value $L(v)$ of the \emph{Minkowski functional} induced by $K$ is defined as the only positive real number such that $v/L(v)\in \partial K$, where $\partial K$ is the boundary of $K$; for the general theory of convex sets and Minkowski spaces see [9] and [15].

\begin{Def}
The function $L$ is called a \emph{Finsler-Minkowski functional} if each non-zero element $v\in \mathbb{R}^n$ has an open neighbourhood such that the restricted function is of class at least $\ \mathcal{C}^4$ and the Hessian matrix 
$$g_{ij}=\frac{\partial^2 E}{\partial y^j \partial y^i}$$
of the energy function $E:=(1/2)L^2$ is positive definite.
\end{Def}

\noindent
II. [1, 10] Since the differentiation decreases the degree of homoge\-ni\-ty in a systematical way we have that $g_{ij}$'s are positively homogeneous of degree zero. To express the infinitesimal measure of changing the inner product it is usual to introduce the (lowered) Cartan tensor by the components
\begin{equation}
\mathcal{C}_{ijk}=\frac{1}{2}\frac{\partial g_{ij}}{\partial y^k}=\frac{1}{2}\frac{\partial ^3 E}{\partial y^k \partial y^j \partial y^i}.
\end{equation}
It is totally symmetric and $y^k\mathcal{C}_{ijk}=0$
because of the zero homogenity of $g_{ij}$. Using its inverse $g^{ij}$ we can introduce the quantities $\mathcal{C}^k_{ij}=g^{kl}\mathcal{C}_{ijl}$ to express the covariant derivative 
$$\nabla_X Y=\left(X^i\frac{\partial Y^k}{\partial y^i}+X^iY^j\mathcal{C}^k_{ij}\right)\frac{\partial}{\partial y^k}$$
with respect to the Riemannian metric $g$. The formula is based on the standard Lévi-Civita process and repeated pairs of indices are automatically summed as usual. The curvature tensor $\mathbb{Q}$ can be written as
$$Q_{ijk}^l=\mathcal{C}^l_{sk}\mathcal{C}^s_{ij}-\mathcal{C}^l_{sj}\mathcal{C}^s_{ik}.$$
Since 
$$y^ig_{ij}=\frac{\partial E}{\partial y^j}=L\frac{\partial L}{\partial y^j},\ \ \textrm{i.e.}\ \ g(C,Y)=YE=L(YL),$$
we have that the Liouville vector field is an outward-pointing unit normal to the indicatrix hypersurface $\partial K:=L^{-1}(1)$. On the other hand
$\nabla_X C=X$ which means that the indicatrix is a totally umbilical hypersurface and $\textrm{div}\  C=n$. 

\vspace{0.2cm}
\noindent
III. Consider the volume form 
$$d\mu=\sqrt{\det g_{ij}}\ dy^1\wedge \ldots \wedge dy^n.$$ 
Let $f$ be a homogeneous function of degree zero. We have by Euler's theorem that
$$\textrm{div}\ (fC)=f\textrm{div}\ C+Cf=nf.$$
Using the divergence theorem (with the Liouville vector field as an outward-pointing unit normal to the indicatrix hypersurface) for the vector field $X=fC$ it follows that
\begin{equation}
\int_K f \, d\mu=\frac{1}{n}\int_{\partial K}f\, \mu,
\end{equation}
where 
$$\mu=\iota_C d\mu=\sqrt{\det g_{ij}}\ \sum_{i=1}^n (-1)^{i-1} y^i dy^1\wedge\ldots\wedge dy^{i-1}\wedge dy^{i+1}\ldots \wedge dy^n$$
 denotes the induced volume form on the indicatrix hypersurface. Using integration we can introduce the following average inner products
\begin{equation}
\gamma_1(v,w):=\int_{\partial K}g(v,w)\, \mu \ \ \textrm{and}\ \ \gamma_2(v,w)=\int_{\partial K} m(v,w)\, \mu_p
\end{equation}
on the vector space, where
$$m(v,w)=g(v,w)-(V L)(W L)$$
is the angular metric tensor and
$$V=v^1\frac{\partial}{\partial y^1}+\ldots+v^n\frac{\partial}{\partial y^n}\ \ \textrm{and}\ \ W=w^1\frac{\partial}{\partial y^1}+\ldots+w^n\frac{\partial}{\partial y^n}.$$
Furthermore
$$\gamma_3(v,w)=\gamma_1(v,w)-\gamma_2(v,w)=\int_{\partial K} (V L)(W L)\, \mu_p.$$
The basic version $\gamma_1$ of the average inner products was introduced in [16]. For further processes to construct inner products by averaging we can refer to [3]. In what follows we take one more step forward by introducing associated Randers-Minkowski functionals. The linear form of the perturbation will play a crucial role in the applications (see Funk metrics and Brickell's theorem).

\begin{Prop}
The linear isometry group of the Minkowski vector space is a subgroup of the orthogonal group with respect to each of the average inner products $\gamma_1$, $\gamma_2$ and $\gamma_3$.
\end{Prop}

\begin{Pf}
Suppose that the linear transformation $\Phi\colon \mathbb{R}^n\to \mathbb{R}^n$ preserves the Minkowskian lenght of the elements in the space. First of all compute the Euclidean volume of $K$:
$$\textrm{vol}_E (K)=\int_K 1=\int_{\Phi(K)}|\det \Phi^{-1}|=|\det \Phi^{-1}| \ \textrm{vol}_E (K)$$
because $K$ is invariant under $\Phi$. Therefore 
$$\det \Phi=\pm 1.$$
Since the energy is also invariant under $\Phi$ we have, by differentiating the relationships 
$$L(\Phi)=L\ \ \ \textrm{and}\ \ \ E(\Phi)=E,$$
respectively, 
that 
$$M_i^k \frac{\partial L}{\partial y^k}= \frac{\partial L}{\partial y^i}\circ \Phi^{-1}\ \ \textrm{and}\ \ M_i^k \frac{\partial E}{\partial y^k}= \frac{\partial E}{\partial y^i}\circ \Phi^{-1},\ \ \textrm{where}\ \ M_i^k=\frac{\partial \Phi^k}{\partial y^i}$$
denotes the matrix of the linear transformation. The second order partial derivatives give that $\Phi$ is an isometry with respect to both the Riemann-Finsler metric and the angular metric tensor: 
$$\ \ M_i^kM_j^lg_{kl}=g_{ij}\circ \Phi^{-1}\ \ \ \textrm{and}\ \ \ M_i^kM_j^lm_{kl}=m_{ij}\circ \Phi^{-1}.$$
Since the absolute value of $\det \Phi$ is $1$ we have that $\det g_{ij}=\det g_{ij}\circ \Phi^{-1}.$
Finally
$$\gamma_1(v,w):=\int_{\partial K}g(v,w)\, \mu=n\int_{K}g(v,w)\, d\mu=n\int_{K}g(v,w) \sqrt{\det g_{ij}}=$$
$$=n\int_{\Phi(K)}g(v,w)\circ \Phi^{-1}\sqrt{\det g_{ij}}\circ \Phi^{-1}|\det \Phi^{-1}|=$$
$$=n\int_{K}g(\Phi(v), \Phi(w))\sqrt{\det g_{ij}}=n\int_{K}g(\Phi(v), \Phi(w))\, d\mu=$$
$$=\int_{\partial K}g(\Phi(v),\Phi(w))\, \mu=\gamma_1(\Phi(v),\Phi(w)).$$
The computation is similar in cases of both the angular metric tensor and the difference metric $\gamma_3=\gamma_1-\gamma_2$.
\end{Pf}

\begin{Cor} If the linear isometry group of the Minkowski vector space is irreducible then there exists a number $\ 0<\lambda < 1$ such that
$\gamma_2=\lambda \gamma_1$ and, consequently, 
$$\gamma_3=(1-\lambda)\gamma_1 \ \ \textrm{and}\ \ \gamma_2=\frac{\lambda}{1-\lambda}\gamma_3.$$
\end{Cor}

\vspace{0.2cm}
\noindent
IV. Using Lemma 1 we have the following practical method to calculate integrals of the form $\int_{\partial K} f\mu$, where the integrand $f$ is a zero homogeneous function. Consider the mapping
$$v \mapsto T(v):=\varphi(v)v,\ \ \textrm{where}\ \ \varphi(v)=\frac{|v|}{L(v)},$$ 
as a diffeomorphism (away from the origin); $|v|$ denotes (for example) the usual Euclidean norm of vectors in $\mathbb{R}^n$. Then
$$\int_{ K} f\, d \mu=\int_{T^{-1}(K)}\varphi^nf\circ \varphi \sqrt{\det g_{ij}}\circ \varphi=\int_{B}\varphi^nf\, d\mu$$
because of the zero homogenity of the integrand. Therefore
\begin{equation}
\int_{B}\varphi^nf\, d\mu=\int_{ K} f\, d \mu
\end{equation}
and, by equation (2),
\begin{equation}
\int_{\partial B}\varphi^nf\, \mu=\int_{\partial K} f\, \mu.
\end{equation} 

\vspace{0.2cm}
\noindent
V. As one of further associated objects consider the linear functional 
$$\beta(v):=\int_{\partial K} VL\, \mu, \ \ \textrm{where}\ \ V=v^1\frac{\partial}{\partial y^1}+\ldots+v^n\frac{\partial}{\partial y^n}.$$

\begin{Rem}\emph{
Using the divergence theorem we have that
$$\int_K \textrm{div}\ \big(LV-(VL)C\big)\, d\mu=0$$
because the vector field $LV-(VL)C$ is tangential to the indicatrix hypersurface:
$$\big(LV-(VL)C\big)L=L(VL)-(VL)L=0.$$
Since $C(VL)=0$
$$\int_K \textrm{div}\ (LV)\, d\mu=\int_K \textrm{div}\ ((VL)C)\, d\mu=n\int_K VL \, d\mu \stackrel{(2)}{=}\int_{\partial K} VL\,  \mu.$$
On the other hand 
$$\textrm{div}\ (LV)=VL+Lv^iC_i,$$
where $C_i=g^{jk}C_{ijk}$.
Since $\textrm{div}\ (LV)$ is a zero homogeneous function
$$n\int_{K} \textrm{div}\ (LV)\, d\mu=\int_{\partial K} \textrm{div}\ (LV)\, \mu=\int_{\partial K} VL+Lv^iC_i\, \mu$$
and thus
$$v^i\int_{\partial K} LC_i\, \mu=(n-1)\int_{\partial K} VL\, \mu=(n-1)\beta(v).$$}
\end{Rem}

\begin{Def}
The indicatrix body $K$ is called \emph{balanced} if the assotiated linear functional $\beta$ is identically zero.
\end{Def}

\begin{Prop}
If the linear isometry group of the Minkowski vector space is irreducible then the indicatrix body is balanced.
\end{Prop}

\begin{Pf}
Using the previous notations
$$\beta(v)=\int_{\partial K} VL\, \mu=n\int_{K} VL\, d\mu=n\int_{K} VL \sqrt{\det g_{ij}}=$$
$$=n\int_{\Phi(K)}v^i\frac{\partial L}{\partial y^i} \circ \Phi^{-1}\sqrt{\det g_{ij}}\circ \Phi^{-1}|\det \Phi^{-1}|=n\int_{K}v^k M_k^i\frac{\partial L}{\partial y^i}\sqrt{\det g_{ij}}=$$
$$=n\int_{K}v^k M_k^i\frac{\partial L}{\partial y^i}\, d\mu=\int_{\partial K}v^k M_k^i\frac{\partial L}{\partial y^i}\, \mu=\beta(\Phi(v)).$$
Therefore the null-space of the linear functional $\beta$ is an invariant subspace under the linear isometry group of the Minkowski vector space. Since the group is irreducible we have that $\beta=0$ as was to be stated.
\end{Pf}

\vspace{0.2cm}
\noindent
VI. Since
$$\beta^2(v)=\left(\int_{\partial K} VL \, \mu\right)^2\leq \int_{\partial K} 1\, \mu\ \int_{\partial K}(VL)^2\, \mu=\textrm{Area}\ (\partial K)\int_{\partial K}(VL)^2\, \mu$$
and
$$(VL)^2=\frac{1}{\ L^2}\ g^2(C,V)\leq \frac{1}{L^2}\ g(C,C)\ g(V,V)=g(V,V)$$
we have that 
$$\beta^2(v)\leq  \textrm{Area}\ (\partial K) \ \gamma_3(v,v) \leq  \textrm{Area}\ (\partial K) \ \gamma_1(v,v).$$
The inequalities are strict because $C$ and $V$ are linearly independent away from the points $v$ and $-v$. Introducing the weighted inner products
$$\Gamma_i(v,w)=\frac{1}{\textrm{Area}\ (\partial K)}\gamma_i(v,w)$$
($i=1, 2, 3$) we can write that
$$\left(\frac{\beta}{\textrm{Area}\ (\partial K)}\right)^2(v)< \Gamma_3(v,v) < \Gamma_1(v,v).$$
Therefore
$$\textrm{the supremum norm of the weighted linear functional}\ \frac{\beta}{\textrm{Area}\ (\partial K)}  <1$$
with respect to each of the unit balls of $\Gamma_1$ and $\Gamma_{3}$. 

\begin{Cor} The functionals
$$L_1(v):=\sqrt{\Gamma_1(v,v)}+\frac{\beta(v)}{\textrm{Area}\ (\partial K)} \ \ \textrm{and}\ \ L_3(v):=\sqrt{\Gamma_3(v,v)}+\frac{\beta(v)}{\textrm{Area}\ (\partial K)}$$
are Finsler-Minkowski functionals which are called associated Randers-Min\-kowski functionals to the Minkowski vector space. 
\end{Cor}

\begin{Thm} \emph{(\bf{The main theorem}).}
If $L$ is a Finsler-Minkowski functional with a balanced indicatrix body of dimension at least three and the Lévi-Civita connection $\nabla$ has zero curvature then $L$ is a norm coming from an inner product, i.e. the Minkowski vector space reduces to a Euclidean one.
\end{Thm}

The original version was proved by F. Brickell [2] using the stronger condition of absolute homogenity (the symmetry of $K$ with respect to the origin) instead of the balanced indicatrix body. The preparations of the proof involve the theory of the Funk spaces and the adaptation of the setting of average objects to Finsler spaces in general. 

\section{Finsler spaces}

\noindent
I. Let $M$ be a differentiable manifold with local coordinates $u^1$, ..., $u^n$ on $U\subset M$. The induced coordinate system on the tangent manifold consists of the functions 
$$x^1:=u^1\circ \pi, \ldots, x^n=u^n\circ \pi\ \ \textrm{and}\ \ y^1:=du^1, \ldots,y^n=du^n,$$ 
where $\pi \colon TM\to M$ is the canonical projection. A Finsler structure on a differentiable manifold $M$ is a smoothly varying family $F\colon TM\to \mathbb{R}$ of Finsler-Minkowski functionals in the tangent spaces satisfying the following conditions:
\begin{itemize}
\item each non-zero element $v\in TM$ has an open neighbourhood such that the restricted function is of class at least $\ \mathcal{C}^4$ in all of its variables $x^1$, ..., $x^n$ and $y^1$, ..., $y^n$,
\item the Hessian matrix of the energy function $E:=(1/2)F^2$ with respect to the variables $y^1, \ldots, y^n$ is positive definite. 
\end{itemize}

\vspace{0.2cm}
\noindent
II. Let $f\colon TM\to \mathbb{R}$ be a zero homogeneous function and let us define the average-valued function
$$A_f(p):=\int_{\partial K_p} f\,  \mu_p,$$
where $\partial K_p$ is the indicatrix hypersurface belonging to the Finsler-Minkow\-ski functional of the tangent space. Integration is taken with respect to the orientation induced by the coordinate vector fields $\partial/\partial y^1, \ldots, \partial/\partial y^n$. It can be easily seen that the integral
$$\int_{K}f\, d\mu_p=\int_{y(K)}f\circ y^{-1}\sqrt{\det g_{ij}}\circ y^{-1}\, dy^1\ldots dy^n$$
is independent of the choice of the coordinate system (orientation). Actually, the orientation is convenient but not necessary to make integrals of functions sense [18].

\noindent
In what follows we compute the partial derivatives of the function $A_f$. Following IV in section 2 we choose an auxiliary Euclidean structure at each point in the coordinate neighbourhood $U$ induced by the coordinate vector fields as an orthonormal system of tangent vectors. Since it is an Euclidean structure there is no need to distinguish the different base points. Using equations (2), (4) and (5)
$$A_f=n\int_{B} \varphi^nf \sqrt{\det g_{ij}}\, dy,\ \ \textrm{where}\ \ \varphi=\frac{\sqrt{(y^1)^2+\ldots+(y^n)^2}}{F}.$$
Therefore
$$\frac{\partial A_f}{\partial u^i}_p=n\int_{B} \frac{\partial}{\partial x^i}\left(\varphi^nf\sqrt{\det g_{mn}}\right)\, dy=$$
$$=n\int_{B} \varphi^n\left(-nf \frac{\partial \ln F}{\partial x^i}+\frac{\partial f}{\partial x^i}+f\frac{\partial \ln \sqrt{\det g_{mn}}}{\partial x^i}\right)\sqrt{\det g_{mn}}\, dy$$
$$=\int_{\partial K_p}-nf \frac{\partial \ln F}{\partial x^i}+\frac{\partial f}{\partial x^i}+f\frac{\partial \ln \sqrt{\det g_{mn}}}{\partial x^i}\, \mu_p$$
using equations (2), (4) and (5) again. Differentiating $\det g_{mn}$ as the composition of the function $D:=\det$ and the usual coordinates of the matrix $M:=g_{ij}$ we have that
$$
\frac{\partial \det g_{mn}}{\partial x^i}=\frac{\partial D}{\ \partial g_{mn}} (M) \frac{\partial g_{mn}}{\partial x^i}=$$
$$=(-1)^{m+n}\det \left(M \ \textrm{ without its $m^{th}$ row and $n^{th}$ column}\right)\frac{\partial g_{mn}}{\partial x^i}=
$$
$$=(\det g_{ij}) g^{mn}\frac{\partial g_{mn}}{\partial x^i}$$
and, consequently,
\begin{equation}
\frac{\partial \ln \sqrt{\det g_{mn}}}{\partial x^i}=\frac{1}{2} g^{mn}\frac{\partial g_{mn}}{\partial x^i}
\end{equation}
The last step is to divide the partial derivatives of the integrand into vertical and horizontal parts [1], see also [5] and [6].

\vspace{0.2cm}
\noindent
III. Using compatible collections  $G_i^k$ of functions on local neighbourhoods of the tangent manifold let us define the (horizontal) vector fields
\begin{equation}
X_i^h=\frac{\partial}{\ \partial x^i}-G_i^k\frac{\partial}{\ \partial y^k}.
\end{equation}
\begin{Def} 
The horizontal distribution $h$ is a collection of subspaces spanned by the vectors $X_i^h$ as the base point runs through the non-zero elements of the tangent manifold. If the functions $G_i^k$ are positively homogeneous of degree $1$ then the distribution is called \emph{homogeneous}. In case of
\begin{equation}
\frac{\partial G_i^k}{\partial y^j}=\frac{\partial G_j^k}{\partial y^i}
\end{equation}
we say that $h$ is \emph{torsion-free}. The horizontal distribution is \emph{conservative} if the derivatives of $F$ vanishes into the horizontal directions 
\end{Def}
It is well-known from the calculus that (8) is a necessary and sufficient condition for the existence of functions $G^k$ such that
$$\frac{\partial G^k}{\partial y^i}=G_i^k.$$ 
Using (7) the integrand can be written into the form
$$-nf \frac{\partial \ln F}{\partial x^i}+\frac{\partial f}{\partial x^i}+f\frac{\partial \ln \sqrt{\det g_{mn}}}{\partial x^i}=-nf X_i^h ln F+X_i^h f+$$
$$+f\frac{1}{2}g^{mn}X_i^hg_{mn}-nfG_i^k\frac{\partial \ln F}{\partial y^k}+G_i^k\frac{\partial f}{\partial y^k}+fG_i^k\overbrace{\frac{1}{2}g^{mn}\frac{\partial g_{mn}}{\partial y^k}}^{{\mathcal{C}_k}}.$$
On the other hand
$$\textrm{div}\ \left(f\left(G_i^k\frac{\partial}{\partial y^k}-G_i^k\frac{\partial \ln F }{\partial y^k}C\right)\right)=f\textrm{div}\ \left(G_i^k\frac{\partial}{\partial y^k}-G_i^k\frac{\partial \ln F }{\partial y^k}C\right)+G_i^k\frac{\partial f}{\partial y^k},$$
where
$$\textrm{div}\ \left(G_i^k\frac{\partial}{\partial y^k}-G_i^k\frac{\partial \ln F }{\partial y^k}C\right)=G_i^k\textrm{div} \frac{\partial}{\ \partial y^k}+\frac{\partial G_i^k}{\partial y^k}-G_i^k\frac{\partial \ln F }{\partial y^k}\ \textrm{div}\ C=$$
$$=G_i^k\mathcal{C}_k+\frac{\partial G_i^k}{\partial y^k}-nG_i^k\frac{\partial \ln F }{\partial y^k}.$$
Using the divergence theorem
$$\int_{\partial K_p} \textrm{div}\ \left(f\left(G_i^k\frac{\partial}{\partial y^k}-G_i^k\frac{\partial \ln F }{\partial y^k}C\right)\right)\, \mu_p=0$$
and thus
$$\frac{\partial A_f}{\partial u^i}=\int_{\partial K_p}-nf X_i^h ln F+X_i^h f+
f\frac{1}{2}g^{mn}X_i^hg_{mn}+f\frac{\partial G_i^k}{\partial y^k}\, \mu_p$$
In terms of index-free expressions 
\begin{equation}
XA_f=\int_{\partial K_p}-nf X^h ln F+X^h f+f\tilde{\mathcal{C}}^{'}(X^c)\, \mu_p,
\end{equation}
where $\tilde{\mathcal{C}}^{'}$ is the semibasic trace of the second Cartan tensor
$$g(\mathcal{C}'(X_i^c,X_j^c),X_k^v)=$$
$$=\frac{1}{2}\left(X_i^h g(X_j^v,X_k^v)-g([ X_i^h,X_j^v],X_k^v)-g(X_j^v,[ X_i^h,X_k^v])\right)$$
associated to $h$, 
 $X^v$, $X^c$ and $X^h$ are the vertical, complete and horizontal lifts of the vector field $X$ on the base manifold. Especially 
$$X_i^v:=\frac{\partial }{\ \partial y^i},\ \ X_i^c:=\frac{\partial }{\ \partial x^i}\ \ \textrm{and}\ \ X_i^h=\frac{\partial}{\ \partial x^i}-G_i^k\frac{\partial}{\ \partial y^k}.$$

\begin{Cor}
If the horizontal distribution is conservative then we have
the reduced formula
\begin{equation}
XA_f=\int_{\partial K_p}X^h f+f\tilde{\mathcal{C}^{'}}(X^c)\, \mu_p.
\end{equation}
\end{Cor}

\noindent
IV. Using the reduced formula (10) we have that
$$X(YA_f)=$$
$$=\int_{\partial K_p}X^h(Y^h f)+\tilde{\mathcal{C}^{'}}(Y^c)X^h f+fX^h\tilde{\mathcal{C}^{'}}(Y^c)+\tilde{\mathcal{C}^{'}}(X^c)Y^h f+f\tilde{\mathcal{C}^{'}}(X^c)\tilde{\mathcal{C}^{'}}(Y^c)\, \mu_p.$$
Changing the role of the vector fields $X$ and $Y$ 
$$[X,Y]A_f=\int_{\partial K_p}[X^h,Y^h] f+f\left(X^h\tilde{\mathcal{C}^{'}}(Y^c)-Y^h\tilde{\mathcal{C}^{'}}(X^c)\right)\, \mu_p.$$
On the other hand
$$[X,Y]A_f\stackrel{(10)}{=}\int_{\partial K_p}[X,Y]^h f+f\tilde{\mathcal{C}^{'}}([X,Y]^c)\, \mu_p$$
and, consequently, we have that
$$
\int_{\partial K_p}R(X^c,Y^c)f\, \mu_p=\int_{\partial K_p}f\left(X^h\tilde{\mathcal{C}^{'}}(Y^c)-Y^h\tilde{\mathcal{C}^{'}}(X^c)-\tilde{\mathcal{C}^{'}}([X,Y]^c)\right)\, \mu_p,
$$
where
$$R(X^c,Y^c):=-v[X^h,Y^h]=[X,Y]^h-[X^h,Y^h]$$
is the curvature of the horizontal distribution, see [5] and [6].

\vspace{0.2cm}
\noindent
V. In case of a conservative and torsion-free horizontal distribution the result can be written into the form
\begin{equation}
\int_{\partial K_p}R(X^c,Y^c)f\, \mu_p=\int_{\partial K_p}f\left((D_{X^h}\tilde{\mathcal{C}^{'}})(Y^c)-(D_{Y^h}\tilde{\mathcal{C}^{'}})(X^c)\right)\, \mu_p,
\end{equation}
where $D$ is one of the Finsler connections of type Berwald, Cartan or Chern-Rund associated to $h$; [13] see also [5] and [6]. Taking $f=1$ we also have that
\begin{equation}
\int_{\partial K_p}\left((D_{X^h}\tilde{\mathcal{C}^{'}})(Y^c)-(D_{Y^h}\tilde{\mathcal{C}^{'}})(X^c)\right)\, \mu_p=0
\end{equation}
provided that $h$ is conservative and torsion-free. 

\vspace{0.2cm}
\noindent
VI. Consider the associated Riemannian metric tensors
$$\gamma_1(X_p,Y_p)=\int_{\partial K_p} g(X^v,Y^v)\, \mu_p,\ \ \gamma_2(X_p,Y_p)=\int_{\partial K_p} m(X^v,Y^v)\, \mu_p,$$
where 
$$m(X^v,Y^v)=g(X^v,Y^v)-(X^vL)(Y^vL)$$
is the angular metric tensor and
$$\gamma_3(X_p,Y_p)=\gamma_1(X_p,Y_p)-\gamma_2(X_p,Y_p)=\int_{\partial K_p} (X^vF)(Y^vF)\, \mu_p$$
together with the weighted versions 
$$\Gamma_i(X_p,Y_p)=\frac{1}{\textrm{Area}\ (\partial K_p)}\gamma_i(X_p,Y_p),$$
where $i=1, 2$, $3$ and  
$$\textrm{Area}\ (\partial K_p)=\int_{\partial K_p} 1\, \mu_p.$$

\begin{Rem}\emph{After introducing the $1$-form
$$\beta(X_p)=\int_{\partial K_p} X^vF\, \mu_p$$
we have, by Corollary 2, that 
$$F_1(v):=\sqrt{\Gamma_1(v,v)}+\frac{\beta(v)}{\textrm{Area}\ (\partial K)} \ \ \textrm{and}\ \ F_3(v):=\sqrt{\Gamma_3(v,v)}+\frac{\beta(v)}{\textrm{Area}\ (\partial K)}$$
are smoothly varying family of Randers-Minkowski functionals which are called the associated Randers functionals to the Finsler space.}
\end{Rem}

\vspace{0.2cm}
\noindent
VII. In what follows we shall use the canonical horizontal distribution of the Finsler manifold which is uniquely determined by the following conditions: it is conservative, torsion-free and homogeneous. Recall the rules of covariant derivative with respect to the Berwald, Cartan and Chern-Rund connections associated to the canonical horizontal distribution:

\vspace{0.5cm}
{\small{\begin{tabular}{|c|c|c|c|}
\hline
&&&\\
 & Berwald& Cartan & Chern-Rund\\
&&&\\
\hline
&&&\\
$D_{X^v}Y^v$ & 0 & $\mathcal{C}(X^c,Y^c)$ & 0\\
&&&\\
\hline
&&&\\
$D_{X^h}Y^v$ & $[X^h,Y^v]$ &  $[X^h,Y^v]+\tilde{\mathcal{C}}'(X^c,Y^c)$ & $[X^h,Y^v]+\tilde{\mathcal{C}}'(X^c,Y^c)$\\
&&&\\
\hline
&&&\\
$D_{X^v}Y^h$ & 0 &  $\mathcal{F}\tilde{\mathcal{C}}(X^c,Y^c)$ & 0\\
&&&\\
\hline
&&&\\
$D_{X^h}Y^v$ & $\mathcal{F}[X^h,Y^v]$ &  $\mathcal{F}[X^h,Y^v]+\mathcal{F}\tilde{\mathcal{C}}'(X^c,Y^c)$ & $\mathcal{F}[X^h,Y^v]+\mathcal{F}\tilde{\mathcal{C}}'(X^c,Y^c)$\\
&&&\\
\hline
\end{tabular}}}

\vspace{0.5cm}
\noindent
where the almost complex structure $\mathcal{F}$ associated to $h$ is defined by the following formulas
$$\mathcal{F}(X^h)=-X^v\ \ \textrm{and}\ \ \mathcal{F}(X^v)=X^h.$$
To determine the Lévi-Civita connections associated to the average Riemannian metrics $\gamma_1$, $\gamma_2$ and $\gamma_3$, respectively, we shall use the Cartan connection because it is metrical, i.e. 
$$X^vg(Y^v,Z^v)=g(D_{X^v}Y^v,Z^v)+g(Y^v,D_{X^v}Z^v),$$
$$
X^hg(Y^v,Z^v)=g(D_{X^h}Y^v,Z^v)+g(Y^v,D_{X^h}Z^v)
$$
and both the (v)v- and the (h)h-torsion vanish. The theory of the classical Finsler connections also says that the deflection tensor is identically zero:
$$
D_{X^h}C=0.
$$ 

\noindent
For the sake of simplicity consider vector fields $X$, $Y$ and $Z$ on the base manifold such that each pair of the vector fields has a vanishing Lie bracket. Using the standard Lévi-Civita proccess it follows by (10) that
$$X\gamma_1(Y,Z)=\int_{\partial K} X^hg(Y^v,Z^v)+\tilde{\mathcal{C}}'(X^c)g(Y^v,Z^v) \, \mu,$$
$$Y\gamma_1(X,Z)=\int_{\partial K} Y^hg(X^v,Z^v)+\tilde{\mathcal{C}}'(Y^c)g(X^v,Z^v) \, \mu,$$
$$Z\gamma_1(X,Y)=\int_{\partial K} Z^hg(X^v,Y^v)+\tilde{\mathcal{C}}'(Z^c)g(X^v,Y^v) \, \mu.$$
Therefore
\begin{equation}
\gamma_1\left((\nabla_1)_{X}Y,Z\right)=\int_{\partial K}g\left(D_{X^h}Y^v,Z^v\right)+\frac{1}{2}\rho(X^c,Y^c,Z^c)\, \mu_,
\end{equation}
where
$$\rho(X^c,Y^c,Z^c)=\tilde{\mathcal{C}^{'}}(X^c)g(Y^v,Z^v)+\tilde{\mathcal{C}^{'}}(Y^c)g(X^v,Z^v)-\tilde{\mathcal{C}^{'}}(Z^c)g(X^v,Y^v).$$
To compute the corresponding formula for $\gamma_3$ we use the identity
$$Y^vF=\frac{1}{F}\ g(Y^v,C).$$
Therefore
$$X^h(Y^vF)=\frac{1}{F}\bigg(g(D_{X^h}Y^v,C)+g(Y^v,D_{X^h}C)\bigg)=\frac{1}{F}\ g(D_{X^h}Y^v,C)$$
because of the vanishing of the deflection.
We have

{\small{
$$X\gamma_3(Y,Z)=\int_{\partial K}g(D_{X^h}Y^v,C)\frac{Z^v F}{F}+\frac{Y^v F}{F}g(D_{X^h}Z^v,C)+\tilde{\mathcal{C}}'(X^c)(Y^v F)(Z^v F) \, \mu,$$
$$Y\gamma_3(X,Z)=\int_{\partial K} g(D_{Y^h}X^v,C)\frac{Z^v F}{F}+\frac{X^v F}{F}g(D_{Y^h}Z^v,C)+\tilde{\mathcal{C}}'(Y^c)(X^v F)(Z^v F) \, \mu,$$
$$Z\gamma_3(X,Y)=\int_{\partial K} g(D_{Z^h}X^v,C)\frac{Y^v F}{F}+\frac{X^v F}{F}g(D_{Z^h}Y^v,C)+\tilde{\mathcal{C}}'(Z^c)(X^v F)(Y^v F) \, \mu,$$}}

\vspace{0.2cm}
\noindent
and the Lévi-Civita process results in the formula
\begin{equation}
\gamma_3\left((\nabla_3)_{X}Y,Z\right)=\int_{\partial K}g\left(D_{X^h}Y^v,C\right)\frac{Z^v F}{F}+\frac{1}{2}\delta(X^c,Y^c,Z^c)\, \mu_p,
\end{equation}
where
$$\delta(X^c,Y^c,Z^c)=\tilde{\mathcal{C}^{'}}(X^c)(Y^v F)(Z^v F)+\tilde{\mathcal{C}^{'}}(Y^c)(X^v F)(Z^v F)-$$
$$\tilde{\mathcal{C}^{'}}(Z^c)(X^v F)(Y^v F).$$

\vspace{0.2cm}
\noindent
Finally $\gamma_2=\gamma_1-\gamma_3$ shows that
$$\gamma_2((\nabla_2)_X Y,Z)=\gamma_1\left((\nabla_1)_{X}Y,Z\right)-\gamma_3\left((\nabla_3)_{X}Y,Z\right).$$
Therefore
\begin{equation}
\gamma_2\left((\nabla_2)_{X}Y,Z\right)=\int_{\partial K}m\left(D_{X^h}Y^v,Z^v\right)+\frac{1}{2}\eta(X^c,Y^c,Z^c)\, \mu_p,
\end{equation}
where 
$$\eta(X^c,Y^c,Z^c)=\tilde{\mathcal{C}^{'}}(X^c)m(Y^v,Z^v)+\tilde{\mathcal{C}^{'}}(Y^c)m(X^v,Z^v)-\tilde{\mathcal{C}^{'}}(Z^c)m(X^v,Y^v).$$

\section{Funk metrics} 

Let $K\subset \mathbb{R}^n$ be a convex body containing the origin in its interior and suppose that the induced function $L$ is a Finsler-Minkowski functional on the vector space $\mathbb{R}^n$. Changing the origin in the interior of $K$ we have a smoothly varying family of Finsler-Minkowski functionals parameterized by the interior points of $K$:
\begin{equation}
L\left(p+\frac{v}{L_p(v)}\right)=1,
\end{equation}
where the script refers to the base point of the tangent vector $v$.

\noindent
I. Conformal geometry. Let $U$ be the interior of the indicatrix body. Together with the Finslerian fundamental function
$$F\colon TU=U\times \mathbb{R}^n\to \mathbb{R},\ v_p\mapsto F(v_p):=L_p(v)$$
the manifold $U$ is called a Funk space (or Funk manifold). It is a special Finsler manifold. 
Another notations and terminology: $K_p$ denotes the unit ball with respect to the functional $L_p$ and $\partial K_p$ is its boundary. Especially $K_{{\bf 0}}=K$ and $\partial K_{{\bf 0}}=\partial K$. Using the Riemann-Finsler metric 
$$g_{ij}=\frac{1}{2}\frac{\partial^2 F^2}{\partial y^i \partial y^j},$$
where the partial derivatives are taken with respect to the variable $v$, any tangent space (away from the origin) can be interpreted as a Riemannian manifold. So is $\partial K_p$ equipped with the usual induced Riemannian structure. It is a totally umbilical Riemannian submanifold of the corresponding tangent space (see section 2/II).
\begin{Thm}
For any $p\in U$ the indicatrix hypersurfaces $\partial K_p$ and $\partial K$ are conformal to each other as Riemannian submanifolds in the corresponding tangent spaces. The conform mapping between these structures comes from the projection
$$\rho(v_p):=p+\frac{v}{L_p(v)}\in \partial K$$
of the tangent space $TU$. Especially
$$g_{\rho(v_p)}(w,z)=\left(1-p^k\frac{\partial L}{\partial u^k}_{\rho(v_p)}\right)g_{v_p}(w,z),$$
where $w$ and $z$ are tangential to the indicatrix hypersurface $\partial K$ at $\rho(v_p)$, i.e they are tangential to $\partial K_p$ at $v_p$ too.
\end{Thm} 

\begin{Pf} Equation (16) says that $L\circ \rho=1$. In terms of the usual local coordinates of the tangent space $TU$ we can write that $\rho^j=x^j+y^j/F$, 
where $\rho^j=u^j\circ \rho$ is the corresponding coordinate function with respect to the standard coordinate system of $U$. Differentiating with respect to $y^i$ we have that
$$0=\frac{\partial L}{\partial u^j}\circ \rho \frac{\partial \rho^j}{\partial y^i}=\frac{\partial L}{\partial u^j}\circ \rho \frac{\delta_i^j F-y^j\frac{\partial F}{\partial y^i}}{F^2}=\frac{1}{F}\frac{\partial L}{\partial u^i}\circ \rho-\frac{y^j}{F^2}\frac{\partial L}{\partial u^j}\circ \rho\frac{\partial F}{\partial y^i}.$$
Uisng the homogenity property
$$\left( x^j+\frac{y^j}{F}\right)\frac{\partial L}{\partial u^j}\circ \rho=L\circ \rho=1$$
we have that
\begin{equation}
0=\frac{1}{F}\left(\frac{\partial L}{\partial u^i}\circ \rho+\left(x^j\frac{\partial L}{\partial u^j}\circ \rho-1\right)\frac{\partial F}{\partial y^i}\right).
\end{equation}
Therefore
\begin{equation}
\frac{\partial L}{\partial u^i}\circ \rho=\left(1-x^k\frac{\partial L}{\partial u^k}\circ \rho\right)\frac{\partial F}{\partial y^i}.
\end{equation}
Differentiating equation (18) with respect to the variable $y^j$  
$$\frac{\partial^2 L}{\partial u^l \partial u^i}\circ \rho\frac{\partial \rho^l}{\partial y^j}=-x^k\frac{\partial^2 L}{\partial u^l \partial u^k}\circ \rho\frac{\partial \rho^l}{\partial y^j}\frac{\partial F}{\partial y^i}+\left(1-x^k\frac{\partial L}{\partial u^k}\circ \rho\right)\frac{\partial^2 F}{\partial y^j \partial y^i},$$
where
$$\frac{\partial \rho^l}{\partial y^j}= \frac{\delta_j^l F-y^l\frac{\partial F}{\partial y^j}}{F^2}$$
and, by the homogenity property,
$$\left( x^l+\frac{y^l}{F}\right)\frac{\partial^2 L}{\partial u^l \partial u^i}\circ \rho=0.$$
Therefore
\begin{equation}
\begin{aligned}
\frac{1}{F}\frac{\partial^2 L}{\partial u^j \partial u^i}\circ \rho&=\left(1-x^k\frac{\partial L}{\partial u^k}\circ \rho\right)\frac{\partial^2 F}{\partial y^j \partial y^i}-\\
\frac{x^l}{F}\frac{\partial^2 L}{\partial u^l \partial u^i}\circ \rho \frac{\partial F}{\partial y^j}&-\frac{x^k}{F}\frac{\partial^2 L}{\partial u^k \partial u^j}\circ \rho \frac{\partial F}{\partial y^i}-\frac{x^k x^l}{F^2}\frac{\partial^2 L}{\partial u^k \partial u^l}\circ \rho \frac{\partial F}{\partial y^i}\frac{\partial F}{\partial y^j}.
\end{aligned}
\end{equation}
Let $w$ and $z$ be tangential to the indicatrix hypersurface $\partial K_p$ at $v_p$, i.e. $F(v_p)=1$ and
$$W_{v_p} F=Z_{v_p} F=0,$$
where
$$W=w^1\frac{\partial}{\partial y^1}+\ldots+w^n\frac{\partial}{\partial y^n}\ \ \textrm{and}\ \ Z=z^1\frac{\partial}{\partial y^1}+\ldots+z^n\frac{\partial}{\partial y^n}.$$
Equation (18) says that $w$ and $z$ are tangential to the indicatrix hypersurface $\partial K$ at $\rho(v_p)$ and, consequently,
$$g_{\rho(v_p)}(w,z)=w^iz^j \frac{\partial^2 L}{\partial u^j \partial u^i}_{\rho(v_p)}.$$
which implies, by equation (19), that 
$$g_{\rho(v_p)}(w,z)=\left(1-p^k\frac{\partial L}{\partial u^k}_{\rho(v_p)}\right)w^iz^j\frac{\partial^2 F}{\partial y^j \partial y^i}_{v_p}=$$
$$\left(1-p^k\frac{\partial L}{\partial u^k}_{\rho(v_p)}\right)g_{v_p}(w,z)$$
as was to be stated. 
\end{Pf}

\begin{Rem} Note that the metric $g$ associated to $L$ is just the Riemann-Finsler metric of the Finsler manifold $(U,F)$ at the origin.
\end{Rem}

\begin{Cor}
$$1-x^k\frac{\partial L}{\partial u^k}\circ \rho>0.$$
\end{Cor}

The geometric meaning of the inequality can be seen from the homogenity property
$$\left( x^j+\frac{y^j}{F}\right)\frac{\partial L}{\partial u^j}\circ \rho=L\circ \rho=1.$$
Especially $v$ is never tangential to the indicatrix hypersurface $\partial K$ at $\rho(v_p)$. Inequality 
$$0<1-x^k\frac{\partial L}{\partial u^k}\circ \rho=\frac{y^j}{F}\frac{\partial L}{\partial u^j}\circ \rho$$
says that $v$ represents an outward pointing direction relative to the indicatrix hyoersurface at $\rho(v_p)$. 

\noindent
II. Projective geometry (Okada's theorem) [11, 12]. After differentiating equation (16) with respect to the variable $x^i$:
$$0=\frac{\partial L}{\partial u^j}\circ \rho \frac{\partial \rho^j}{\partial x^i}=\frac{\partial L}{\partial u^j}\circ \rho \left(\delta_i^j-\frac{y^j}{F^2}\frac{\partial F}{\partial x^i}\right)=\frac{\partial L}{\partial u^i}\circ \rho-\frac{y^j}{F^2}\frac{\partial L}{\partial u^j}\circ \rho \frac{\partial F}{\partial x^i}$$
Using the homogenity property
$$\left( x^j+\frac{y^j}{F}\right)\frac{\partial L}{\partial u^j}\circ \rho=L\circ \rho=1$$
again we have that
$$0=\frac{\partial L}{\partial u^i}\circ \rho+\frac{1}{F}\left(x^j\frac{\partial L}{\partial u^j}\circ \rho-1\right)\frac{\partial F}{\partial x^i}.$$
Together with equation (17) and Corollary 4 
\begin{equation}
0=\frac{\partial F}{\partial y^i}-\frac{1}{F}\frac{\partial F}{\partial x^i} \ \Rightarrow\ F\frac{\partial F}{\partial y^i}=\frac{\partial F}{\partial x^i}
\end{equation}
which is just Okada's theorem for Funk spaces. In terms of differential geometric structures we can write equation (20) into the form
$$d_h F=Fd_JF,$$
where $h$ is the horizontal distribution detemined by the coordinate vector fields $\partial/\partial x^1$, ..., $\partial/\partial x^n$ and $J$ is the canonical almost tangent structure on the tangent space $TU$. Then
$$d_Jd_hF=d_JF \wedge d_J F+Fd_J^2F=0$$
which is just the coordinate-free expression for the classical Rapcsák's equation of projective equivalence [13]. Therefore the canonical spray $\xi$ of the Funk space can be expressed into the following special form
$$\xi=y^i\frac{\partial}{\partial x^i}-Fy^i\frac{\partial}{\partial y^i}.$$
It says that the straight lines in U can be reparameterized to the geodesics of the Funk manifold [8]. If $c(t)=p+tv$ then we need the solution of the differential equation\footnote{Equation (21) can be considered as a correction of equation (21) in [17].}
\begin{equation}
\theta''=-\theta'F(v_p)
\end{equation}
Under the initial conditions $\theta(0)=0$ and $\theta'(0)=1$. We have that
\begin{equation}
\theta(t)=\frac{1-\textrm{exp}\ (-tF(v_p))}{F(v_p)},
\end{equation}
i.e.
$$\tilde{c}(t)=p+\frac{1-\textrm{exp}\ (-tF(v_p))}{F(v_p)}v$$
is a geodesic of the Funk manifold. Okada's theorem is a rule how to change derivatives with respect to $x^i$ and $y^i$. This results in relatively simple formulas for the canonical objects of the Funk manifold. In what follows we are going to summarize some of them (proofs are straightforward calculations, [1]):
\begin{equation}
G^k=\frac{1}{2}y^kF,\ \ G_i^k=\frac{\partial G^k}{\partial y^k}=\frac{F}{2}\delta_{i}^k+\frac{1}{2}y^k \frac{\partial F}{\partial y^i},
\end{equation}
\begin{equation}
X_i^h=\frac{\partial}{\ \partial x^i}-G_i^k\frac{\partial}{\ \partial y^k}=\frac{\partial}{\ \partial x^i}-\left(\frac{F}{2}\delta_{i}^k+\frac{1}{2}y^k \frac{\partial F}{\partial y^i}\right)\frac{\partial}{\ \partial y^k}
\end{equation}
for the canonical horizontal distribution. The first and the second Cartan tensors are related as  
\begin{equation}
\mathcal{C}'=\frac{1}{2}F\mathcal{C}
\end{equation}
and the curvature of the canonical horizontal distribution can be expressed in the following form 
\begin{equation}
R\left(\frac{\partial}{\ \partial x^i},\frac{\partial}{\ \partial x^i}\right)=\frac{F}{4}\left(\frac{\partial F}{\ \partial y^i}\frac{\partial}{\ \partial y^j}-\frac{\partial F}{\ \partial y^j}\frac{\partial}{\ \partial y^i}\right).
\end{equation}
In terms of lifted vector fields
\begin{equation}
R(X^c,Y^c)=\frac{1}{4}\bigg{(}g(X^v,C)Y^v-g(Y^v,C)X^v\bigg{)}.
\end{equation}

\noindent
III. Applications. Using relation (25) we can specialize the basic formula (10) for derivatives of integral functions. We apply this technic only in the special case of the area-function
\begin{equation}
r\colon U\to \mathbb{R}, \ \ p\mapsto r(p):=A_1(p)=\int_{\partial K_p} 1\, \mu_p
\end{equation}

\begin{Thm} The area function is convex.
\end{Thm}

\begin{proof}
Remark 1 shows that
\begin{equation}
\frac{\partial r}{\partial u^i}_p =\frac{1}{2}\int_{\partial K_p} F\tilde{\mathcal{C}}\left(\frac{\partial}{\partial x^i}\right)\, \mu_p=\frac{n-1}{2}\int_{\partial K_p} \frac{\partial F}{\partial y^i}\, \mu_p.
\end{equation}
Differentiating again
\begin{equation}
\frac{\partial^2 r}{\partial u^j \partial u^i}_p=\frac{n-1}{2}\int_{\partial K_p} \left(\frac{\partial}{\partial u^j}\right)^h \frac{\partial F}{\partial y^i}+\frac{1}{2} \frac{\partial F}{\partial y^i}F\tilde{\mathcal{C}}\left(\frac{\partial}{\partial x^j}\right)\, \mu_p.
\end{equation}
Here, by (24),
$$\left(\frac{\partial}{\partial u^j}\right)^h \frac{\partial F}{\partial y^i}=\frac{\partial^2 F}{\partial x^j \partial y^i}-\frac{1}{2}F\frac{\partial^2 F}{\partial y^j \partial y^i}$$
because of the zero homogenity of the vertical partial derivatives of the Finslerian fundamental function. Differentiating Okada's relation (20) it follows that
$$\frac{\partial^2 F}{\partial x^j \partial y^i}=\frac{\partial F}{\partial y^j}\frac{\partial F}{\partial y^i}+F\frac{\partial^2 F}{\partial y^j \partial y^i}$$
and thus
$$\left(\frac{\partial}{\partial u^j}\right)^h \frac{\partial F}{\partial y^i}=\frac{\partial F}{\partial y^j}\frac{\partial F}{\partial y^i}+\frac{1}{2}F\frac{\partial^2 F}{\partial y^j \partial y^i}.$$
The only question is how to substitute the last term in the integrand of (30) to give a positive definite expression. First of all note that the vector field
$$T:=F\frac{\partial F}{\partial y^i}\frac{\partial}{\partial y^j}-\frac{\partial F}{\partial y^i}\frac{\partial F}{\partial y^j}C$$
is tangential to the indicatrix hypersurface and, by the divergence theorem,
$$\int_{K_p} \textrm{div}\ T\, d\mu_p=0.$$
In a more detailed form (using the zero homogenity of the vertical derivatives of the Finslerian fundamental function)
$$\textrm{div}\ T=F\frac{\partial F}{\partial y^i} \ \textrm{div}\frac{\partial}{\partial y^j}+\frac{\partial F}{\partial y^i}\frac{\partial F}{\partial y^j}+F\frac{\partial^2 F}{\partial y^j \partial y^i}-\frac{\partial F}{\partial y^i}\frac{\partial F}{\partial y^j}\ \textrm{div}\ C=$$
$$= F \frac{\partial F}{\partial y^i}\  \tilde{\mathcal{C}}\left(\frac{\partial}{\partial x^j}\right)+F\frac{\partial^2 F}{\partial y^j \partial y^i}-(n-1) \frac{\partial F}{\partial y^i}\frac{\partial F}{\partial y^j}.$$
We have
$$\int_{K_p}F\frac{\partial F}{\partial y^i}\  \tilde{\mathcal{C}}\left(\frac{\partial}{\partial x^j}\right)\, d\mu_p=
\int_{K_p} (n-1) \frac{\partial F}{\partial y^i}\frac{\partial F}{\partial y^j}-F\frac{\partial^2 F}{\partial y^j \partial y^i}\, d\mu_p$$
and thus
\begin{equation}
\frac{\partial^2 r}{\partial u^j \partial u^i}_p=\frac{n^2-1}{4}\int_{\partial K_p} \frac{\partial F}{\partial y^i}\frac{\partial F}{\partial y^j}\,\mu_p.
\end{equation}
The convexity follows from the positive definiteness of the Hessian matrix of the area function.
\end{proof}

\begin{Rem}
Equation \emph{(}29\emph{)} implies that 
$$\textrm{d} r=\frac{n-1}{2}\beta$$
and, consequently, the associated Randers functionals can be written into the following forms
$$F_1(v):=\sqrt{\Gamma_1(v,v)}+\frac{2}{n-1}\textrm{d}\log r \ \ \textrm{and}\ \ F_3(v):=\sqrt{\Gamma_3(v,v)}+\frac{2}{n-1}\textrm{d}\log r.$$
Thereore they are projectively equivalent to the Riemannian manifold U equipped with the weighted Riemannian metric tensors $\Gamma_1$ and $\Gamma_3$, respectively; see \emph{[7]}. On the other hand equation \emph{(}31\emph{)} shows that the second order partial derivatives of the area function coincide the associated Riemannian metric $\gamma_3$ up to a constant proportional term.
\end{Rem}

\begin{Cor} The Minkowskian indicatrix is ballanced if and only if the area function of the associated Funk space has a global minimizer at the origin.
\end{Cor}

To finish this section recall that the projection $\rho$ provides a nice connection between the indicatrix hypersurfaces $\partial K_p$ and $\partial K$. Using theorem 2 we have
$$r(p)=\int_{\partial K_p}\, \mu_p=\int_{\partial K} \left(1-p^k\frac{\partial L}{\partial u^k}\right)^{-\frac{n-1}{2}}\, \mu,$$
where $\mu=\mu_{\bf{0}}$ is the canonical volume form associated to $K$. From the general theory of Minkowski functionals [1]
$$w^k\frac{\partial L}{\partial u^k}_v \leq L(w)$$
for any nonzero element $v$. Substituting $p\in U$ as $w$ we have that
$$1-p^k\frac{\partial L}{\partial u^k}\geq 1-L(p)>0$$
In case of $w=-p$ 
$$1-p^k\frac{\partial L}{\partial u^k}\leq 1+L(-p)$$
and we have the following estimations:
\begin{equation}
\bigg(1+L(-p)\bigg)^{-\frac{n-1}{2}}\leq \frac{r(p)}{r({\bf{0}})} \leq \bigg(1-L(p)\bigg)^{-\frac{n-1}{2}},
\end{equation}
where
$$r({\bf{0}})=\int_{\partial K} 1\, \mu.$$

\section{The proof of the main theorem}

\begin{Thm} \emph{(\bf{The main theorem}).}
If $L$ is a Finsler-Minkowski functional with a balanced indicatrix body of dimension at least three and the Lévi-Civita connection $\nabla$ has zero curvature then $L$ is a norm coming from an inner product, i.e. the Minkowski vector space reduces to a Euclidean one.
\end{Thm}

\begin{Pf}
Since the indicatrix body is ballanced it follows from equation (29) that the area function has zero partial derivatives at the origin. The convexity implies that it has a global minimum at $\bf{0}$ which is just the area of the Euclidean sphere of dimension $n-1$ because of the condition of the vanishing of the curvature (recall that the dimension is at least three which means that we have a simply connected Riemannian manifold with zero curvature - it is isometric to the standard Euclidean space). If $p$ denotes the centroid of $K$ then the area of $\partial K_{p}$ is at least $r({\bf 0})$ and, by the generalized Santalo's inequality [4] (for bodies with center at the origin) $K_{p}$ is a Euclidean ball. Since $K$ is a translate of $K_p$ the functional $L$ associated to $K$ is a Randers-Minkowski functional with vanishing curvature tensor $\mathbb{Q}$. This is impossible unless the translation is trivial, i.e. $p$ is just the origin and $K$ is an Euclidean ball as was to be stated.
\end{Pf}

\section*{Acknowledgements}
This work was partially supported by the European Union and the European Social Fund through the project Supercomputer, the national virtual lab (grant no.:T\'AMOP-4.2.2.C-11/1/KONV-2012-0010)

This work is supported by the University of Debrecen's internal research project.

\end{document}